\documentclass{article}
\title{A counterexample to Kanalas' problem of continuously realising types}
\author{Morgan Rogers\footnote{Université Sorbonne Paris Nord, LIPN, CNRS, UMR 7030, F-93430, Villetaneuse, France, \texttt{rogers@lipn.univ-paris13.fr}} \ and 
	Joshua Wrigley\footnote{Department of Mathematics and Statistics, Faculty of Sciences, Masaryk University, Kotlářská 2, 611 37 Brno, Czech Republic, \texttt{wrigley@math.muni.cz}}}
\date{}

\usepackage{amsmath,amsthm,amssymb}
\usepackage{tikz-cd}
\usepackage{quiver}
\usepackage{thmtools}
\usepackage[hidelinks]{hyperref}

\usepackage{enumitem}
\setlist{listparindent = \parindent, parsep=0pt,}
\setenumerate[1]{label = (\roman*), ref = (\roman*)}

\usepackage[noabbrev,capitalise]{cleveref}
\crefformat{equation}{(#2#1#3)}
\crefformat{enumi}{#2#1#3}
\crefformat{enumii}{#2#1#3}

\usepackage[a4paper,
left=1.6in,
right=1.6in,
top=1in,
bottom=1in
]{geometry}

\theoremstyle{plain}
\newtheorem{thm}{Theorem}[section]

\newtheorem{coro}[thm]{Corollary}
\newtheorem{prop}[thm]{Proposition}
\newtheorem{question}[thm]{Question}

\theoremstyle{definition}
\newtheorem{df}[thm]{Definition}

\renewcommand{\phi}{\varphi}
\newcommand{\cat}{\mathcal{C}}
\newcommand{\dcat}{\mathcal{D}}
\newcommand{\pretopos}{\mathcal{P}}
\newcommand{\Sh}{\mathsf{Sh}}
\newcommand{\topos}{\mathcal{E}}
\newcommand{\ftopos}{\mathcal{F}}
\newcommand{\Set}{\mathbf{Set}}
\newcommand{\locrefl}{\mathcal{L}}
\newcommand{\theory}{\mathbb{T}}
\newcommand{\opens}{\mathcal{O}}
\newcommand{\op}{^{\mathrm{op}}}

\usepackage[
backend=biber,
sorting=nyt,
style=ieee
]{biblatex}
\begin{document}
	
	\maketitle
	
	\begin{abstract}
		We give an example of a coherent theory, a topological space, and a continuous assignment of the (model-theoretic) types of the theory to the points of the space, such that there is no sheaf model of the theory over the space with the property that the fibre over a point, as a set-based model, realises the corresponding type; thus, we answer in the negative a problem posed by Krist\'of Kanalas.
	\end{abstract}
	
	\renewcommand{\thefootnote}{\fnsymbol{footnote}} 
	\footnotetext{
		\emph{Key words:} Realising types, sheaf model, localic reflection.
	}
	\footnotetext{
		\emph{2020 Mathematics subject classification:} 03C90, 03G30, 18F10, 54B40.
	}     
	\renewcommand{\thefootnote}{\arabic{footnote}} 
	
	\section{Introduction}
	In a recent paper \cite[Question 3.6]{kanalas_apal}, Krist\'of Kanalas posed the question: 
	\begin{question}\label{question_1}
		Does every Grothendieck topos $\topos$ `realise types' of coherent theories?
	\end{question}
	We will answer this question in the negative using a finite counterexample.  First, let us explain the idea of `realising types' in a sheaf topos, and why this question is a natural one in the context of sheaf semantics of first-order theories.
	
	Let $\theory$ be a coherent theory (\cite[Definition D1.1.6]{elephant}, also called a \emph{positive theory} by some).  Recall that the sentences of $\theory$, i.e.\ the coherent formulae with no free variables, naturally form a distributive lattice $\dcat$.  Let $S_0(\theory)$ denote the dual spectral space of $\dcat$.  A point $p$ of $S_0(\theory)$ is called a positive $0$-type over the empty parameter set by model theorists.  Recall that the point $p$ corresponds to a \emph{prime filter} of $\dcat$; by completeness of $\theory$ with respect to set-based models, we may also equate the point $p$ with the set of all the sentences satisfied by some set-based model $M$ of $\theory$.  That is, $p$ describes a model up to (positive) elementary equivalence.
	
	Since the only types we will encounter in this paper are the positive $0$-types over the empty parameter set, we will use the word `type' to exclusively refer to these.
	\begin{df}\label{df:realise}
		We say that a set-based model $M$ \emph{realises} a type $p$ if the set of sentences true in $M$ {contains} $p$.  (Note that we are asking for containment, and not equality, of sets of sentences: the model $M$ can satisfy more sentences than those contained in $p$.)
	\end{df}

	Now suppose we have a topological space $X$, which for convenience we assume is \emph{sober}, and to each point $x \in X$ we associate a type in a continuous fashion, i.e.\ we are given a continuous function $q \colon X \to S_0(\theory)$.  A very reasonable question is if we can associate a model $M_x$ to each point $x \in X$, also in a continuous fashion, such that the model $M_x$ realises the type $q(x) \in S_0(\theory)$. In this instance, `continuous fashion' means that the models $M_x$ are the fibres of a local homeomorphism $M \to X$, and moreover every $n$-ary definable relation is an open subset of the corresponding $n$-fold pullback $M \times_X \dots \times_X M$.  In other words, there is a model $M$ \emph{internal} to the topos of sheaves $\Sh(X)$ on $X$ such that, for each $x \in X$, the fibre $M_x$ as a model realises the type $q(x)$ (see \cite[\S D1.2]{elephant} for more on internal models).  In the talks \cite{kanalas_talk_3,kanalas_talk_2}, Kanalas motivated this question as a form of `compactness theorem' for sheaf topoi.
	
	We now reformulate the question in the language of categorical logic, which is the way that the question is formulated in \cite{kanalas_apal} or \cite[Defintion 8.8]{kanalas_phd}.  The coherent theory $\theory$ can be replaced by its \emph{classifying pretopos} $\pretopos$ (as in \cite[\S 8.4]{makkai_reyes}), and the distributive lattice $\dcat$ of sentences becomes the lattice of subterminals of $\pretopos$.  By sobriety, the continuous assignment of types $q \colon X \to S_0(\theory)$ is equivalent in data to a homomorphism of distributive lattices $\dcat \to \opens(X)$ to the lattice of open subsets of $X$, or equivalently a coherent functor $\dcat \to \Sh(X)$.  A model of $\theory$ internal to $\Sh(X)$ is equivalent in data to a coherent functor $\pretopos \to \Sh(X)$.  The question of whether the topos $\Sh(X)$ realises types for the theory $\theory$ thus becomes:
	
	\begin{question}\label{question_2}
		For every coherent functor $Q \colon \dcat \to \Sh(X)$, does there exist a coherent functor $M \colon \pretopos \to \Sh(X)$ and a natural transformation
		\[\begin{tikzcd}
			\dcat & \\
			\pretopos & {\Sh(X)} \, ?
			\arrow[hook, from=1-1, to=2-1]
			\arrow[""{name=0, anchor=center, inner sep=0}, "Q", from=1-1, to=2-2]
			\arrow["M"', dashed, from=2-1, to=2-2]
			\arrow["\Rightarrow"{marking, allow upside down}, draw=none, from=0, to=2-1]
		\end{tikzcd}\]
	\end{question}
	
	The existence of the natural transformation $Q \Rightarrow M|_{\dcat}$ expresses that the model at a point $x \in X$ realises the type $q(x)$.
	\begin{df}[Definition 3.1 \cite{kanalas_apal}]
		If the topos $\Sh(X)$ has the above property for every pretopos, we say that $\Sh(X)$ \emph{realises types}.
	\end{df}
	By replacing the topos $\Sh(X)$ with an arbitrary Grothendieck topos $\topos$, we can just as easily ask if $\topos$ realises types.  We now have enough terminology to appreciate Kanalas' problem: {\it does every topos $\topos$ realise types for every pretopos?}

	Kanalas has demonstrated that certain special classes of topoi do realise types: if $\topos$ is the topos of sheaves on a complete Boolean algebra with the coherent topology \cite[Theorem 3.5]{kanalas_phd}, or if $\topos$ is the topos of presheaves on well-founded tree \cite{kanalas_talk_2}.

	We will demonstrate a counterexample to \cref{question_1} and \cref{question_2} by constructing the following:
	\begin{thm}
		There is a poset $P\cat$ and a (related) coherent theory $\mathbb{T}$ such that the topos of co-presheaves $[P\cat, \Set]$ does not realise types for $\mathbb{T}$.
	\end{thm}

	First, we will reformulate Kanalas' problem to an equivalent statement which first engendered our doubts to its validity.  Recall that each pretopos $\pretopos$ has a \emph{coherent} classifying topos $\topos$, so that geometric morphisms $\ftopos \to \topos$ correspond to coherent functors $\pretopos \to \ftopos$ (\cite[\S 8]{makkai_reyes}).  Under this correspondence, the identity on $\topos$ corresponds to the inclusion $\pretopos \hookrightarrow \topos$.  The coherent topos $\topos$ has a \emph{localic reflection} $\locrefl$ (\cite[\S A4.6]{elephant}), which classifies the lattice $\dcat$ of subterminals of $\pretopos$, in the sense that geometric morphisms $\ftopos \to \locrefl$ correspond to coherent functors $\dcat \to \ftopos$.  The canonical hyperconnected morphism $h \colon \topos \to \locrefl$ of the localic reflection corresponds to the inclusion $\dcat \hookrightarrow \pretopos \hookrightarrow \topos$.
	\begin{prop}
		Let $\pretopos$ be a pretopos with classifying topos $\topos$.  Every Grothendieck topos realises types for $\pretopos$ if and only if the localic reflection $h \colon \topos \to \locrefl$ admits a lax section, as in the following diagram.
		\[\begin{tikzcd}
			\locrefl & \topos \\
			& \locrefl
			\arrow[dashed, from=1-1, to=1-2]
			\arrow[""{name=0, anchor=center, inner sep=0}, equals, from=1-1, to=2-2]
			\arrow["h", from=1-2, to=2-2]
			\arrow["\Rightarrow"{marking, allow upside down}, draw=none, from=0, to=1-2]
		\end{tikzcd}\]
	\end{prop}
	\begin{proof}
		Recall that $\dcat$ embeds into its classifying topos $\locrefl$, and recall also that this inclusion $\dcat \hookrightarrow \locrefl$ corresponds, under the classifying property of $\locrefl$, to the identity morphism on $\locrefl$.  If $\locrefl$ realises types, then there must be a natural transformation
		\[\begin{tikzcd}
			\dcat & \\
			\pretopos & \locrefl,
			\arrow[hook, from=1-1, to=2-1]
			\arrow[""{name=0, anchor=center, inner sep=0}, hook, from=1-1, to=2-2]
			\arrow[dashed, from=2-1, to=2-2]
			\arrow["\Rightarrow"{marking, allow upside down}, draw=none, from=0, to=2-1]
		\end{tikzcd}\]
		which, via converting to classifying topoi, yields the desired lax splitting
		\[\begin{tikzcd}
			\locrefl & \\
			\topos & \locrefl.
			\arrow[""{name=0, anchor=center, inner sep=0}, equals, from=1-1, to=2-2]
			\arrow["h", from=2-1, to=1-1]
			\arrow[dashed, from=2-2, to=2-1]
			\arrow["\Rightarrow"{marking, allow upside down}, draw=none, from=0, to=2-1]
		\end{tikzcd}\]
		
		Now suppose we are given a lax section as above.  Then for any other topos $\ftopos$ and a coherent functor $\dcat \to \ftopos$, we take the corresponding geometric morphism $\ftopos \to \locrefl$ and then form the following pasting of geometric morphisms.
		\[\begin{tikzcd}
			\ftopos & \locrefl & \topos \\
			&& \locrefl
			\arrow[from=1-1, to=1-2]
			\arrow[from=1-2, to=1-3]
			\arrow[""{name=0, anchor=center, inner sep=0}, equals, from=1-2, to=2-3]
			\arrow["h", from=1-3, to=2-3]
			\arrow["\Rightarrow"{marking, allow upside down}, draw=none, from=0, to=1-3]
		\end{tikzcd}\]
		Now, converting back to the coherent functors that are being classified, we obtain a diagram of coherent functors
		\[\begin{tikzcd}
			\ftopos & \pretopos \\
			& \dcat
			\arrow[from=1-2, to=1-1]
			\arrow[""{name=0, anchor=center, inner sep=0}, from=2-2, to=1-1]
			\arrow[hook, from=2-2, to=1-2]
			\arrow["\Rightarrow"{marking, allow upside down}, draw=none, from=0, to=1-2]
		\end{tikzcd}\]
		as desired.
	\end{proof}
	Suddenly, the seemingly reasonable suggestion that we can continuously assign models extending a continuous assignment of types, for any coherent theory, becomes the rather more dubious assertion that the localic reflection of any coherent topos has a lax section.  This is the starting point of finding our counterexample.
	%%%%%%%%%%%%%%%%%%%%%%%%%%%%%%%%%%%%%%%%%%%%%%%%%%%%%%%%%%%%%%%%%%%%%%%%%%%%%
	\section{The counterexample}
	Consider the following category $\cat$ with all non-identity morphisms displayed.
	\[
	\begin{tikzcd}[sep = large]
		{A } & B & C \\
		D & E \\
		F
		\arrow["\kappa", from=1-1, to=1-2]
		\arrow["\lambda"', from=1-1, to=2-1]
		\arrow["{\mu \circ \kappa}", shift left = 1.5, from=1-1, to=2-2]
		\arrow["{\nu \circ \lambda}"', shift right = 1.5, from=1-1, to=2-2]
		\arrow[from=1-2, to=1-3]
		\arrow["\mu", from=1-2, to=2-2]
		\arrow["\nu"', from=2-1, to=2-2]
		\arrow[from=2-1, to=3-1]
	\end{tikzcd}
	\]
	In particular, the composites $\mu \circ \kappa$ and $\nu \circ \lambda$ are \emph{not} equal.  The topos of co-presheaves $[\cat,\Set]$ is coherent since $\cat$ is finite.  Phrased in another way, the topos $[\cat,\Set]$ classifies a coherent theory, namely the theory of flat-functors on $\cat\op$, which is indeed coherent since $\cat$ is finite (see the axiomatisation of this theory given \cref{df:flat_functors} or \cite[Theorem 2.1.11]{TST}).
	
	\begin{prop}\label{prop:C_is_cauchy_complete}
		The category $\cat$ is Cauchy-complete.
	\end{prop}
	\begin{proof}
		By direct inspection, there are no non-identity idempotents.
	\end{proof}
	\begin{coro}\label{coro:geom_morphs_are_functors}
		For any other category $\dcat$, a geometric morphism $[\dcat,\Set] \to [\cat,\Set]$ corresponds to a functor $\dcat \to \cat$.  Similarly, a transformation of geometric morphisms corresponds to a natural transformation of functors.
	\end{coro}
	\begin{proof}
		Using \cref{prop:C_is_cauchy_complete} and \cite[Expos\'e VI, Exercice 3.12]{SGA} (cf.\ also \cite[Lemma A4.1.5]{elephant}).
	\end{proof}
	By \cite[Example C2.3.20]{elephant}, the localic reflection of $[\cat,\Set]$ is given by the topos $[P\cat,\Set]$ of co-presheaves on the \emph{posetal reflection} $P\cat$ of $\cat$, that is, the following poset, displayed as a category with all non-identity arrows.
	\[\begin{tikzcd}
		{A } & B & C \\
		D & E \\
		F
		\arrow[from=1-1, to=1-2]
		\arrow[from=1-1, to=2-1]
		\arrow[from=1-1, to=2-2]
		\arrow[from=1-2, to=1-3]
		\arrow[from=1-2, to=2-2]
		\arrow[from=2-1, to=2-2]
		\arrow[from=2-1, to=3-1]
	\end{tikzcd}\]
	Note that the canonical hyperconnected morphism $h \colon [\cat,\Set] \to [P\cat,\Set]$ of the localic reflection is witnessed by the canonical functor $\eta \colon \cat \to P\cat$ of the posetal reflection.  This also describes the action of $h$ on points of $[\cat,\Set]$ (under \cref{coro:geom_morphs_are_functors}).
	
	Thus, by \cref{coro:geom_morphs_are_functors}, to find a lax section of $h \colon [\cat,\Set] \to [P\cat,\Set]$, we must find a functor $\theta \colon P \cat \to \cat$ and a natural transformation as below.
	\[\begin{tikzcd}
		{P \cat} & \cat \\
		& {P \cat}
		\arrow["\theta", from=1-1, to=1-2]
		\arrow[""{name=0, anchor=center, inner sep=0}, equals, from=1-1, to=2-2]
		\arrow["\eta", from=1-2, to=2-2]
		\arrow["\Rightarrow"{marking, allow upside down}, draw=none, from=0, to=1-2]
	\end{tikzcd}\]
	In the context of posets, the existence of the required natural transformation expresses that $X \leqslant \eta \theta(X) $ for all $X \in P\cat$.
	\begin{thm}\label{thm:counterexample}
		There is no functor $\theta \colon P\cat \to \cat$ such that $X \leqslant \eta \theta (X)$.
	\end{thm}
	\begin{proof}
		Suppose such a functor $\theta \colon P \cat \to \cat$ did exist.  Since $C,E,F$ are maximal elements of $P \cat$, we must have $C = \eta \theta (C)$, $E = \eta \theta (E)$ and $F = \eta \theta (F)$.  Next, $\eta \theta (D)$ must be a lower bound to $\eta \theta (F)$ and $\eta \theta (E)$ by functoriality, and also larger than or equal to $D$, which forces $\eta \theta (D) = D$.  Similarly, we deduce that $\eta \theta (B) = B$ and $\eta \theta (A) = A$.  Hence $\eta \theta$ acts as the identity on objects, and so this must also be true of $\theta$.
		
		For most of the arrows of $P \cat$ there is exactly one choice for where $\theta$ must send them. Consequently, for functoriality we must have
		\begin{align*}
			\mu \circ \kappa &= \theta (B \to E) \circ \theta (A \to B) \\
			& = \theta (A \to B \to E) = \theta (A \to E) = \theta (A \to D \to E) \\
			& = \theta (D \to E) \circ \theta (A \to D) = \nu \circ \lambda,
		\end{align*}
		which is a contradiction. Thus the identity map on objects cannot be extended to a functor, as claimed.
	\end{proof}
	\paragraph*{The counterexample topologically.}
	To conclude, we give a second exposition of the same counterexample, but in the concrete topological terms of `continuously realising types' as narrated in the introduction, for the benefit of the reader unfamiliar with the techniques of topos theory; that is to say, we explicitly describe the theory of flat functors on $\cat\op$, and its corresponding topological space of types.
	
	The following is an equivalent axiomatisation of the theory of flat functors on $\cat\op$, adapted from, for instance, \cite[Theorem 2.1.11]{TST}.
	\begin{df}\label{df:flat_functors}
		We denote by $\theory_{\cat\op}$ the following coherent theory.
		\begin{enumerate}
			\item The language of our theory $\theory_{\cat\op}$ has six sorts, $A, B, \dots , F$, and six function symbols:
			\begin{align*}
				\kappa & \colon B \to A , & \mu & \colon D \to A, & \pi & \colon C \to B , \\
				\lambda & \colon D \to A, & \nu & \colon E \to D,& \xi & \colon F \to D.
			\end{align*}
			\item We add an axiom $\top \vdash (\exists \, a : A )\lor (\exists \, b : B ) \lor \dots \lor (\exists \, f : F)$, expressing that one of the sorts, $A, B, \dots , F$, is non-empty.
			\item We add an axiom $b_1, b_2 : B \vdash b_1 = b_2$, expressing that the sort $B$ has at most one element.  We add the same axiom for the sorts $C, D, E, F$ (but not the sort $A$).
			\item We add the axioms $a_1, a_2 : A \vdash a_1 = a_2 \lor \exists\, e : E .(\kappa(\mu(e)) = a_1 \land \lambda(\nu(e)) = a_2)$ and $\exists\, e: E . (\kappa(\mu(e))=\lambda(\nu(e))) \vdash \bot$ expressing that the sort $A$ either has at most one element, or the sort $E$ contains a (necessarily unique) element $e$ and $A$ consists precisely of the two distinct images $\kappa(\mu(e))$ and $\lambda(\nu(e))$.
			\item Finally, we add an axiom $(\exists \, b : B) \land (\exists \, d : D) \vdash \exists \, e : E$, expressing that if the sorts $B$ and $D$ are non-empty then so too is $E$, and the axioms
			\[
			(\exists \, c : C)\land (\exists \, e : E) \vdash \bot, \quad (\exists\, c : C) \land (\exists \, f : F) \vdash \bot, \quad (\exists \, e : E) \land (\exists \, f : F) \vdash \bot,
			\]
			expressing that, between $C, E, F$, at most one of these sorts is non-empty.
		\end{enumerate}
	\end{df}
	It is not difficult to write down all the possible models for this theory (there are six up to isomorphism); for instance, one model interprets the sorts as: $A = \{a_1,a_2\}$, $B=\{b\}$, $C=\emptyset$, $D=\{d\}$, $E=\{e\}$, $F = \emptyset$.  Each of these six models corresponds to one of the six distinct types for the theory, which are all \emph{supported} in the sense of \cite[Definition 3.1.2]{kamsma_txtbk}; namely, the types are $\langle \exists\, a : A \rangle, \dots , \langle \exists\, f : F \rangle $ (where $\langle \phi \rangle $ denotes $ \{ \psi \in \dcat \mid \phi \vdash \psi \}$, i.e.\ the set of sentences entailed by $\phi$).  For notational simplicity, we will use $\underline{a}$ to denote the type $\langle \exists \, a : A\rangle$ viewed as a point its space of types.  The space of types for $\theory_{\cat\op}$ is given by the following topological space $X$, where all non-empty opens are displayed.
	\begin{center}
		\begin{tikzpicture}
			%nodes
			\node (A) at (2,0) {$\underline{a}$};
			\node (B) at (1,0) {$\underline{b}$};
			\node (C) at (0,0) {$\underline{c}$};
			\node (D) at (2,1) {$\underline{d}$};
			\node (E) at (1,1) {$\underline{e}$};
			\node (F) at (2,2) {$\underline{f}$};
			%opens at C, E and F
			\draw (C) circle (0.3);
			\draw (E) circle (0.3);
			\draw (F) circle (0.3);
			%open EDF
			\draw
			($(D)+(0.4,0)$) -- ($(F) + (0.4,0)$) arc (0:135:0.4) -- ($(E)+(135:0.4)$) arc (135:270:0.4) -- ($(D)+(0,-.4)$) arc (270:360:0.4);
			%open CBE
			\draw[color = white, line width = 5]
			($(B)+(0.5,0)$) -- ($(E) + (0.5,0)$) arc (0:135:0.5) -- ($(C)+(135:0.5)$) arc (135:270:0.5) -- ($(B)+(0,-.5)$) arc (270:360:0.5);
			\draw
			($(B)+(0.5,0)$) -- ($(E) + (0.5,0)$) arc (0:135:0.5) -- ($(C)+(135:0.5)$) arc (135:270:0.5) -- ($(B)+(0,-.5)$) arc (270:360:0.5);
			%open all
			\draw
			($(A)+(0.6,0)$) -- ($(F) + (0.6,0)$) arc (0:135:0.6) -- ($(C)+(135:0.6)$) arc (135:270:0.6) -- ($(A)+(0,-.6)$) arc (270:360:0.6);
		\end{tikzpicture}
	\end{center}
	Note that the topos $[P \cat, \Set]$ is precisely the topos of sheaves $\Sh(X)$ over this space.
	
	Recall that, by a \emph{sheaf model} of $\theory_{\cat\op}$ over $X$, we mean a system consisting of a local homeomorphism for each sort, $q_A \colon U_A \to X$, etc., and a morphism of local homeomorphisms for each function symbol, i.e.\ a continuous map $q_\kappa \colon U_B \to U_A$ such that $q_A \circ q_\kappa = q_B$, etc., such that, for each point $x \in X$, the fibres over $x$ describe a model of $\theory_{\cat\op}$.  \cref{thm:counterexample} demonstrates that there is no sheaf model of $\theory_{\cat\op}$ over the space $X$ such that model over the point $\underline{a}$ realises $\exists \, a : A$, etc.  This can also be proved using a direct topological argument using the concrete descriptions of $\theory_{\cat\op}$ and $X$ provided above; it is a simple exercise in topology, which we encourage the reader to attempt.  We briefly outline a potential strategy below.
	\begin{enumerate}
		\item Suppose that such a sheaf model did exist.  First, using the fact that each sort must be interpreted by a local homeomorphism, show that the fibre over a point $x \in X$ must be the model that \emph{strictly} realises the corresponding type, in the sense that the set of sentences true in the model is precisely equal to the prime filter described by the type (cf.\ \cref{df:realise}), i.e.\ the model over $\underline{c}$ has $A , B, C$ as singleton sets and $E , D, F$ empty, etc.
		\item Deduce that, as a topological space, the space $U_A$ interpreting the sort $A$ must be homeomorphic to the disjoint union $X + \{\underline{e}\}$.
		\item Then show that there is a unique choice of commuting continuous maps $q_\xi \colon U_F \to U_D$, etc., for each function symbol.  Finally, observe that these unique choices satisfy $q_\kappa \circ q_\mu = q_\lambda \circ q_\nu$, and so the constructed system of local homeomorphisms is not a model of $\theory_{\cat\op}$.
	\end{enumerate}
	\subsubsection*{Acknowledgements}
	The second author acknowledges the support from Marie Sk{\l}odowska-Curie Grant No.\ 101273434.
	
	\smallskip
	\rightline{\includegraphics[height=1.0cm]{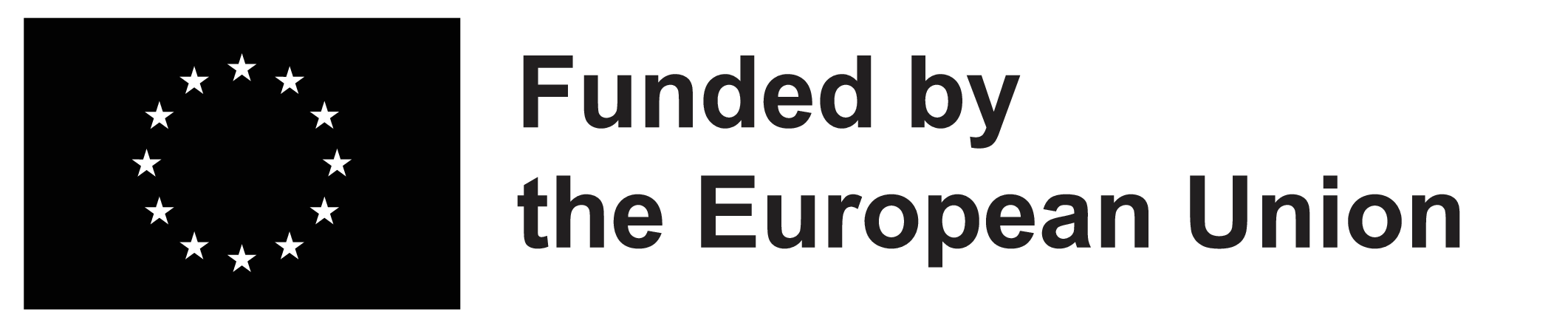}}
	
	%References
	\printbibliography
\end{document}